\documentclass[1p]{elsarticle}

\usepackage[latin1]{inputenc}

\journal{Journal of \LaTeX\ Templates}

\usepackage{graphicx}
\usepackage{color}

\usepackage{amssymb}
\usepackage{eucal}
\newcommand{\R}{{\Bbb R}}

\newcommand{\N}{{\Bbb N}}
\newcommand{\C}{{\Bbb C}}

\usepackage{color}

\usepackage{amsmath}

\newtheorem{thm}{Theorem}
\newtheorem{lemma}[thm]{Lemma}
\newtheorem{corollary}[thm]{Corollary}
\newtheorem{proposition}[thm]{Proposition}

\newtheorem{remark}[thm]{Remark}
\newtheorem{example}[thm]{Example}
\newproof{proof}{Proof}

\begin{document}

\begin{frontmatter}
\title{Existence and uniqueness of monotone wavefronts in a nonlocal resource-limited model }

\author[a]{ Elena Trofimchuk\,}
\author[b]{ Manuel Pinto\,}
\author[c]{     Sergei Trofimchuk\footnote{Corresponding author.}}
\address[a]{Department of Differential Equations, Igor Sikorsky Kyiv Polytechnic Institute, Kyiv, Ukraine
\\ {\rm E-mail: trofimch@imath.kiev.ua}}
\address[b]{Facultad de Ciencias, Universidad de Chile, Santiago, Chile
\\ {\rm E-mail: pintoj@uchile.cl}}
\address[c]{Instituto de Matem\'atica y F\'isica, Universidad de Talca, Casilla 747,
Talca, Chile \\ {\rm E-mail: trofimch@inst-mat.utalca.cl}}

\begin{abstract}
We are revisiting the topic of travelling  fronts  for the food-limited (FL) model  with spatio-temporal nonlocal reaction. 
 These solutions are crucial for understanding the whole model dynamics. Firstly, we prove the existence of monotone wavefronts. In difference with all previous results formulated in terms of `sufficiently small parameters',  our existence theorem indicates 
 a reasonably broad and explicit range of the model key parameters allowing the existence of monotone waves.  Secondly,  numerical simulations  realized on the base of our analysis show appearance of non-oscillating and non-monotone travelling fronts  in the FL model. These waves were never observed before. Finally, invoking a new approach developed recently by Solar {\it et al}, 
we prove the uniqueness (for a fixed propagation speed, up to translation) of each monotone front. 
 \end{abstract}


\end{frontmatter}


\section{Introduction}

In this work, we consider  the following   food-limited (FL, for short) model \cite{FL,Gopal, GG,GC, OW,PR,Sm,SY,WL,Pr} with spatio-temporal nonlocal reaction 
\begin{equation}\label{i1}
\partial_tu(t,x)=\partial_{xx}u(t,x) +u(t,x)\left(\frac{1-(K *u)(t,x)}{1+\gamma (K*u)(t,x)}\right), \quad x \in \mathbb R.
\end{equation}
Here
$$
(K*u)(t,x) = \int_0^{+\infty}ds \int_\R K(s,y,\tau)u(t-s, x-y)dy, 
$$
and $K:\R_+\times\R \times\R_+  \to \R_+$ is a function satisfying
$$
\frak{I}_{c,\tau}(\lambda):= \int_{\R_+\times\R}  K(s,y,\tau)e^{-\lambda(cs+y)}dsdy \in \R, \ \  \frak{I}_{c,\tau}(0)=1, \ \lim_{\lambda\to- \lambda_0(c)^+}\frak{I}_{c,\tau}(\lambda)=+\infty, 
$$
for all $\lambda \in (-\lambda_0(c), \lambda_1(c)), \ c >0$, and some  $\lambda_0(c), \lambda_1(c) \in (0,+\infty]$. Parameter $\tau  >0$ takes into account the delayed effects in the model  and it is related to the average maturation time. Note that we admit the distributional kernels $K$ (as in Proposition \ref{PRO}), however, we will assume that 
$\frak{I}_{c,\tau}(\lambda)$ is a usual scalar continuous function of variables  $c,\tau, \lambda$.  

In view of the above mentioned references, model  (\ref{i1}) seems to be among the most studied equations of  population dynamics. It can be also viewed as a natural extension of the nonlocal KPP-Fisher  equation  \cite{AB,BNPR,FZ,GT,GF,HK,wz} which is obtained from $(\ref{i1})$ by letting $\gamma=0$. Positive smooth solutions $u(t,x)=\phi(x+ct)$ of  (\ref{i1}) satisfying the boundary conditions $\phi(-\infty)=0, \ \phi(+\infty) =1$  (and usually called wavefront solutions, or simply wavefronts)  are key elements for understanding the whole 
evolution process governed by the food-limited equation. The parameter $c$ is called the wave speed, it is a well known fact that $c\geq 2$ (cf. Subsection \ref{S22} below).  Starting from the pioneering works by Gourley \cite{GG} and Gourley and Chaplain \cite{GC}, 
the existence and monotonicity properties of wavefronts  for the FL model have been  the  object of investigation in  various papers, e.g. see  \cite{OW,WL,Pr}.  From these works, we have the following consolidated existence result. 
\begin{proposition} \label{PRO} Suppose that the kernel $K$ has one of two forms
\begin{equation}\label{K1a}
K(s,y,\tau)= \delta(y)G_1(s,\tau),  \quad  K(s,y,\tau)= \frac{e^{-y^2/4s}}{\sqrt{4\pi s}}G_2(s,\tau),
\end{equation}
where either  $G_1(s,\tau)=\delta(s-\tau)$ (local in space, discrete delay case \cite{GG,WL}),  or $G_2(s,\tau)=\delta(s-\tau)$ (nonlocal in space, discrete delay case \cite{OW}),  or $G_1(s,\tau) =(s/\tau^2) e^{-s/\tau}$ (local in space, strong generic delay case \cite{OW,WL}), or 
$G_2(s,\tau) =(s/\tau^2) e^{-s/\tau}$ (nonlocal in space, strong generic delay case \cite{OW,WL,Pr}), or $G_2(s,\tau) =(1/\tau) e^{-s/\tau}$ (nonlocal in space, weak generic delay case \cite{GC}). 

Then for each $c \geq 2$ there exists $\tau(c)$ such that equation (\ref{i1}) with $\tau \in (0, \tau(c)]$ has at least one wavefront propagating with the speed $c$. 

Moreover, in the case when  $G_2(s,\tau)=\delta(s-\tau)$ and $\tau < (1+\gamma)3/2$, there is $c(\tau,\gamma)\geq 2$ such that (\ref{i1})  has at least one wavefront for each  speed $c \geq c(\tau, \gamma)$ \cite{OW}. 
\end{proposition}
Several relevant remarks are in order. First, in view of  approaches used in the  cited works, 
the numbers $\tau(c)$ and $1/ c(\tau, \gamma)$ in Proposition \ref{PRO} have to be {\it sufficiently} small.  On the other hand, 
derivation of explicit lower estimates for them  seems to be rather problematic. At the same time, available upper estimations of $\tau(c)$ do not exclude that it may happen that  $\tau(c) \to 0$ as $c\to \infty$, see \cite{GG,WL} or related argumentation  in \cite{HT}.  Certainly, all this limits the applicability of Proposition \ref{PRO}. Second, the monotonicity of obtained wave profiles $\phi(s)$ was proved only for the particular cases considered in \cite{GG,WL}. The studies \cite{GG,GC, OW} present numerical simulations showing that $\phi(s)$ is monotone for small $\tau$ and that it can oscillate around the equilibrium $1$ if $\tau$ is sufficiently large.  And third, 
important and difficult problem of the uniqueness of wavefronts for the FL model  was not discussed in the mentioned articles. 

In  the present work,  inspired by  recent significant advances in the studies of nonlocal KPP-Fisher equation  \cite{AC,BNPR,DN,FZ,GT,HK,HT,NP} (i.e. of equation (\ref{i1}) with  $\gamma=0$), 
we are going to clarify all  three aforementioned aspects concerning the FL model.  We will also discuss remarkable differences existing between the cases $\gamma =0$ and 
$\gamma >0$. We explain why, even in more simple case of monotone wavefronts for (\ref{i1}), it seems impossible to obtain a concise criterion, similar to the Fang and Zhao 
criterion \cite{FZ}, of their existence when  $\gamma \gg 1$.

Previously, three main approaches were used to prove Proposition \ref{PRO}: a) the Zou and Wu monotone iterations method \cite{wz}, extended for nonlocal systems in \cite{WLR}. It allows to consider rather general kernels $K$ and prove the monotonicity of waves,  see \cite{GG,WL}; b) the linear chain technique combined with Fenichel's invariant manifold theory. Only special kernels are admitted by this approach,   see \cite{GC,Pr}; c) the Hale-Lin perturbation techniques  \cite{FH,HL,HK} based on the use of the Lyapunov-Schmidt reduction of the profile equation  in appropriate infinite-dimensional spaces.  In this case the non-perturbed equation (when $\tau=0$ or $1/c=0$) has a wavefront and it is shown that it can be extended continuously for small $\tau>0$ or large $c \gg 2$. As in a), this method can also be applied to rather general forms of $K$, see \cite{OW}.  

In our analysis of the wavefront existence problem for the FL model, we are using an appropriate modification of the Zou and Wu monotone iterations algorithm from \cite{wz}. The recent works \cite{FZ,GT,HT,TPT} have shown high efficiency of  this method in the studies of delayed \cite{GT,TPT}, nonlocal \cite{FZ} and neutral \cite{HT} versions of the KPP-Fisher equation as well as of the nonlocal diffusive equation  of the Mackey-Glass type \cite{TPT}. In this regard, this article provides a natural extension, for $\gamma >0$, of the findings in \cite{FZ,GT}  containing them as very particular cases. 

Clearly,  the existence of  classical travelling wave solution $u(x,t)=\phi(x+ct)$ amounts to the existence of positive solution $\phi(t)$ to the following boundary value problem
\begin{equation}\label{f1}
 \phi''(t) - c\phi'(t) +\phi(t)\left(\frac{1-(N_c *\phi)(t)}{1+\gamma (N_c*\phi)(t)}\right)=0, \quad  \phi(-\infty) = 0 , \;\; \phi( +\infty)=1, 
\end{equation}
where
$$
N_c(v,\tau): = \int_0^{+\infty}K(s,v-cs,\tau)ds, \quad \mbox{so that} \  \int_\R N_c(v,\tau)e^{-\lambda v}dv = \frak{I}_{c,\tau}(\lambda). 
$$
Note that each non-constant solution $\phi(t) \in [0,1], \ t \in \R,$ of (\ref{f1})  satisfies   $\phi'(t) >0$ and $\phi(t) \in (0,1)$ for all $t \in\R$, see Lemmas \ref{mono} and \ref{USAI} below. 

After linearizing (\ref{f1}) around the positive equilibrium, we get the characteristic function $\chi_+$ determining  the asymptotic behavior of wavefronts at $+\infty$, 
$$
\chi_+(z,c,\tau) = z^2-cz - \frac{1}{1+\gamma} \int_\R N_c(v,\tau)e^{-z v}dv. 
$$
Now we are in position to state the first main result of this work. 
\begin{thm}\label{T11}  If, for some $c' \geq 2$, $\tau>0$, the  equation $\chi_+(z,c',\tau)=0$ has a negative root $z'$  satisfying the inequality 
{$\gamma(z'^2 -cz') \leq 1$ }  then 
 equation (\ref{i1}) has a monotone  wavefront $u(x,t)=\phi(x+c't)$, $\phi'(s) >0, \ s \in \R$. Conversely,  if  equation (\ref{i1}) has a monotone wavefront propagating with some 
 speed $c'$ then $c' \geq 2$ and there exists a negative real number $z'$ such that $\chi_+(z',c',\tau)=0$. \end{thm}
To compare this result with Proposition  \ref{PRO}, let consider  the second kernel in (\ref{K1a}).  We have 
\begin{equation}\label{K1}
\int_\R N_c(v,\tau)e^{-\lambda v}dv = \int_0^{+\infty}G_2(s,\tau)e^{(\lambda^2-\lambda c)s}ds \in \R_+, 
\end{equation}
so that 
$$
\chi_+(z,c,\tau) =\chi(z^2-cz,\tau), \quad \mbox{where} \ \chi(w,\tau)= w -  \frac{1}{1+\gamma}\int_0^{+\infty}G_2(s,\tau)e^{ws}ds. 
$$
Here is an immediate consequence of the  above computation and Theorem \ref{T11}.
\begin{corollary}\label{C1} Let the kernel be as in (\ref{K1a}) with $G_2(s,\tau)$ and $\gamma >0$. Then, if equation  (\ref{i1})  
has a monotone wavefront propagating with the speed  $c'$, $\chi$ has a  positive zero $w'$.
 Conversely, suppose that $\chi$ has a zero  $w' \in (0, 1/\gamma)$. 
Then for each $c' \geq 2$ equation  (\ref{i1})  has a monotone wavefront propagating with the speed  $c'$.  
\end{corollary}
\begin{example}\label{ex1} As in \cite{GC}, consider weak generic delay kernel  $G_2(s,\tau) =(1/\tau) e^{-s/\tau}$.  Then  
$$
\chi(w,\tau)= w + \frac{1}{(1+\gamma)(w\tau-1)}, \quad w < 1/\tau,
$$
and after invoking  Corollary \ref{C1}, it is easy to find that the inequality $\tau \leq (1+\gamma)/4$ is a necessary condition for the existence of monotone wavefronts. 
By the same  corollary, a straightforward analysis shows that  monotone wavefronts exist when 
 \begin{equation}\label{ex1c}
 \tau \leq \left\{
\begin{array}{ccc}
(1+\gamma)/4 & ,& \gamma \in (0,1],\\
\gamma/(1+\gamma) & , & \gamma \geq 1.  
\end{array} \right. 
\end{equation}
In difference with the result provided by Proposition \ref{PRO}, the upper estimate for $\tau$ in (\ref{ex1c})  is explicit and does not vanish as $c \to +\infty$. Moreover, for values of  $\gamma \in (0,1]$, the inequality $\tau \leq (1+\gamma)/4$  gives a sharp criterion for the existence of monotone wavefronts.    To understand what can happen for $\gamma >1$, following \cite{GC}, we  applied the linear chain technique rewriting equation  (\ref{i1})  as the system of two coupled reaction-diffusion equations
\begin{equation}\label{s1}
u_t=du_{xx} +u\left(\frac{1-v}{1+\gamma v}\right), \quad v_t=dv_{xx} +\frac{1}{\tau}(u-v). 
\end{equation}
Here, similarly to \cite{GC}, we introduced a diffusivity parameter $d$ to have a  greater control of the propagation speed (in any case, notice that condition (\ref{ex1c}) neither depends on the diffusivity $d$ nor on the speed $c$). Next, we realized numerical simulations with MATLAB taking $\gamma =40, \ \tau = 10$, $d =50$ and considering (\ref{s1}) on the interval $[0,1800]$ with homogeneous Neumann boundary conditions and initial conditions 
$$u(0,x) = v(0,x) = \left\{
\begin{array}{ccc}
0 & ,& t \leq 1500,\\
1& , & t > 1500.
\end{array} \right. $$
With the above indicated values, we have $10=\tau \leq (1+\gamma)/4 =41/4$, however, $\gamma(dz'^2 -cz')= \gamma w' = 2(1-1/\sqrt{41})>1$. As Figure  \ref{non} shows, 
the solution $u(t,x)$ of the problem converges to a wavefront propagating with the minimal speed $c = 10\sqrt{2}$. 
This wavefront is clearly non-monotone, with  its maximal value being bigger than $1$.  Nevertheless, as the analysis in \cite{IGT} suggests, the wavefront   neither is oscillating  around the positive equilibrium (actually, all zeros of $\chi_+(z,10\sqrt{2},10)$ are real). 
\begin{figure}[ht!]
\centering\fbox{{\includegraphics[scale=0.38]{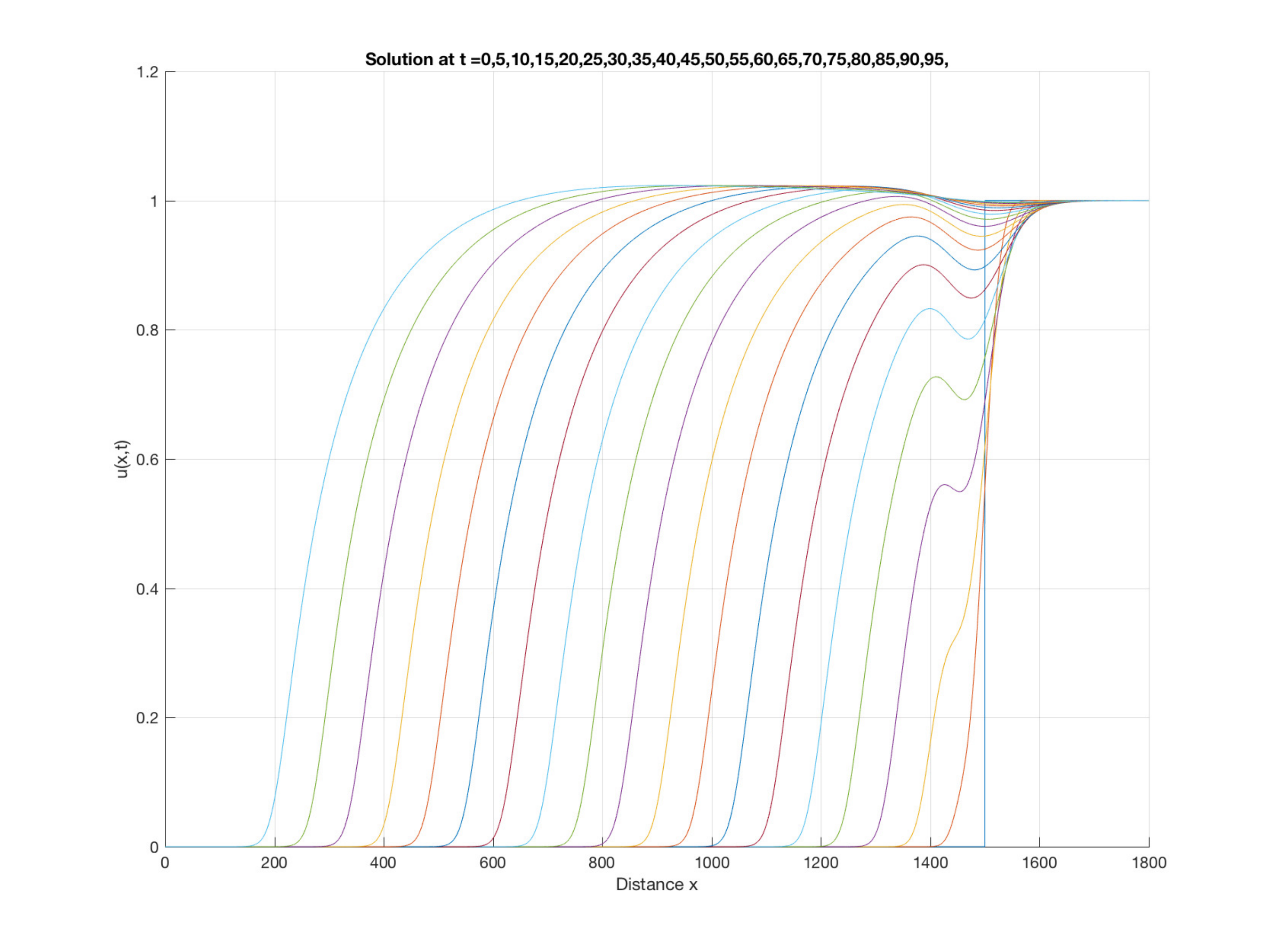}}}
\caption{Development of a minimal non-monotone  non-oscillating wavefront in the nonlocal FL model with the weak generic delay kernel.}
\label{non}
\end{figure}
It seems that such a dynamical behavior (the existence of a non-monotone but eventually monotone wavefront in  the FL model) was not observed in the previous numerical experiments, cf. \cite{GG,GC,OW}.  From the analysis in \cite{IGT} we see that the reason of this phenomenon can be explained by the fact that the nonlinear function
\begin{equation}\label{F}
F(u,v) := u\left(\frac{1-v}{1+\gamma v}\right)
\end{equation}
with $\gamma >0$
does not satisfy the sub-tangency condition at the positive steady state $$F(u,v) \leq F_u(1,1)(u-1)+  F_v(1,1)(v-1) = \frac{1-v}{1+\gamma}, \ u,v \in (0,1).$$ 
This fact also indicates that the problem of determining  necessary and sufficient conditions for the existence  of monotone wavefronts in (\ref{i1}) with $\gamma >0$ might not have 
an explicit solution in terms of the parameters $\tau, \gamma, c$. Here we have a full analogy with the very difficult problem of determining speed of propagation of the pushed wavefronts, see \cite{BD} and references therein. 
\end{example}  
\begin{example}\label{ex2} As in  \cite{OW,WL,Pr}, consider the strong generic delay kernel  
$G_2(s,\tau) =(t/\tau^2) e^{-s/\tau}$. Then $$
\chi(w,\tau)= w - \frac{1}{(1+\gamma)(w\tau-1)^2}, \quad w < 1/\tau,
$$
and the sufficient condition of Corollary \ref{C1} for the existence of monotone wavefronts takes the form 
  $$ \tau \leq \left\{
\begin{array}{ccc}
4(1+\gamma)/27 & ,& \gamma \in (0,4/5],\\
\gamma(1-\sqrt{\frac{\gamma}{1+\gamma}}) & , & \gamma \geq 4/5.  
\end{array} \right. $$ 
Moreover, the inequality $\tau \leq 4(1+\gamma)/27$ is a necessary condition for the existence of monotone wavefronts.  All other conclusions and discussion of Example \ref{ex1} are also apposite  to  this particular case. 
\end{example}
\begin{example}\label{C2} Finally, as in \cite{GG}, we consider the FL model  with a single discrete delay and without nonlocal interaction, i.e. when the kernel in (\ref{K1a}) is chosen with  $G_1(s,\tau)= \delta(t-\tau)$ and  $\gamma >0$. In several aspects, this case appears to be more difficult than the situations considered in Examples \ref{ex1}, \ref{ex2}.  As we have already 
mentioned, it does not allow the use of the linear chain technique.  It is immediate  to find that  
 $$
\chi_+(z,c,\tau) = z^2-cz - \frac{e^{-z ch}}{1+\gamma}. 
$$
Applying Theorem \ref{T11} and after some computation,  we see that  a monotone wavefront exists for  each $c \geq 2$ if 
$$0\leq \tau\leq \gamma\ln\frac{1+\gamma}{\gamma}.
$$
But even if the latter inequality is not satisfied, the monotone  wavefronts still exist for the speeds 
$$
2 \leq c \leq \frac{\ln\frac{1+\gamma}{\gamma}}{\tau \sqrt{\frac 1 \gamma - \frac 1 \tau \ln\frac{1+\gamma}{\gamma}}}. 
$$
In Figure \ref{stab1}, the regions of parameters $(\tau,c) \in [0,+\infty)\times [2+\infty)$ bounded by the blue lines and the coordinate axes correspond to some subdomains of existence 
of monotone wavefronts (determined by the sufficient condition of Corollary \ref{C1}). The domains  bounded by the red lines and the coordinate axes correspond to the domains of the existence 
of eventually monotone wavefronts (necessary condition of Corollary \ref{C1}). 
\begin{figure}[ht!]
\centering\fbox{ \hspace{-5mm} {\includegraphics[scale=0.24]{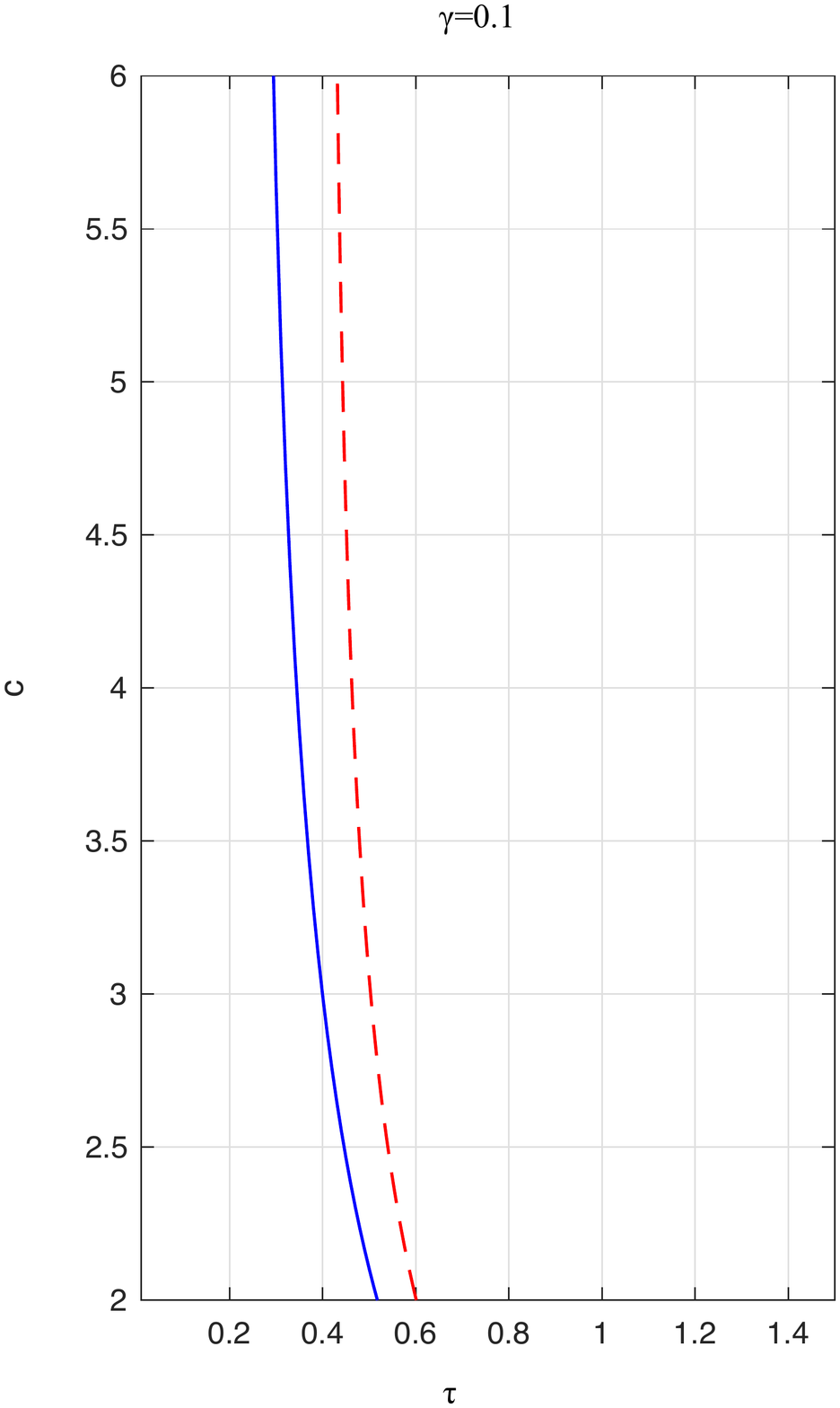}} \hspace{-15mm} {\includegraphics[scale=0.24]{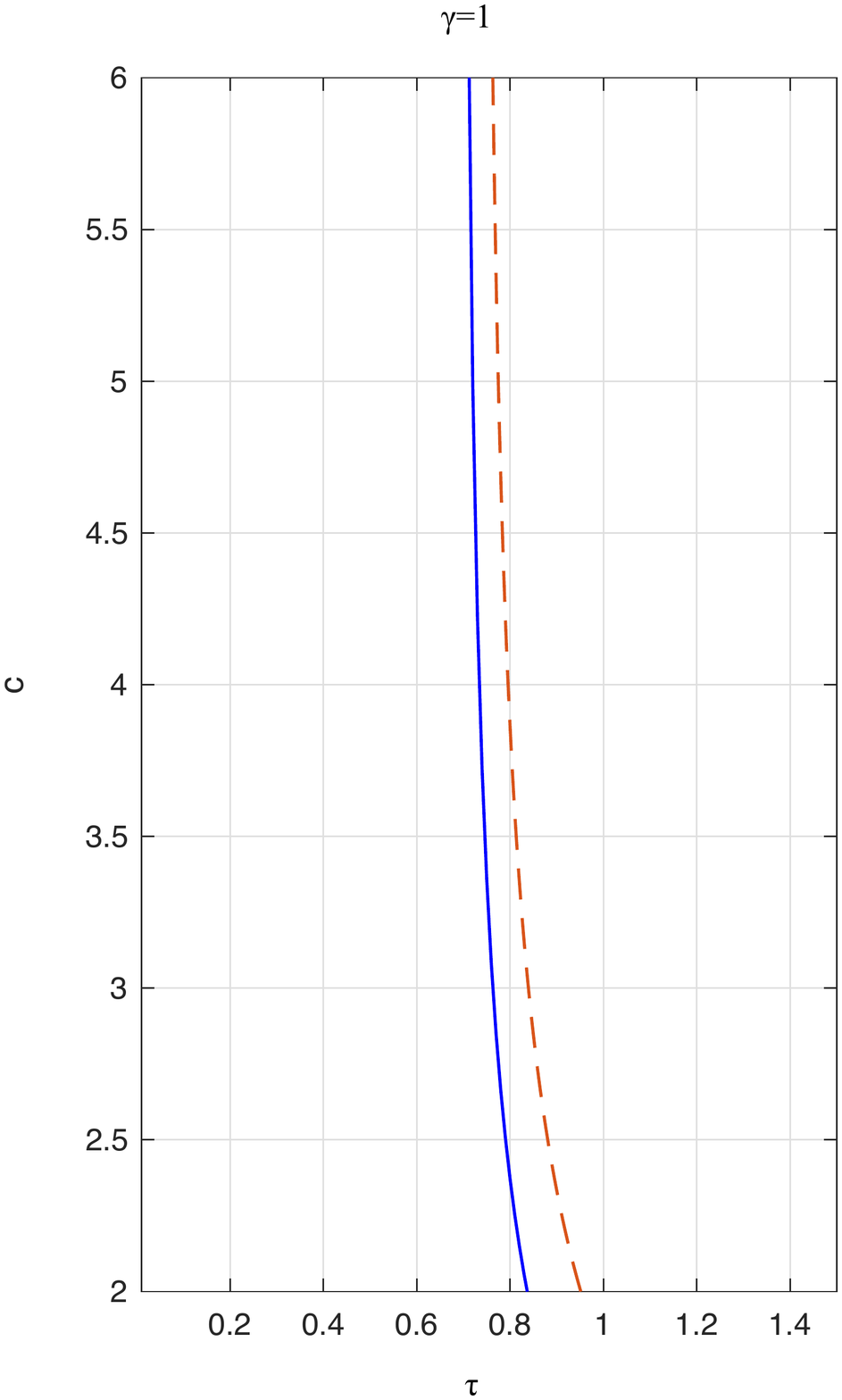}} \hspace{-15mm}  {\includegraphics[scale=0.24]{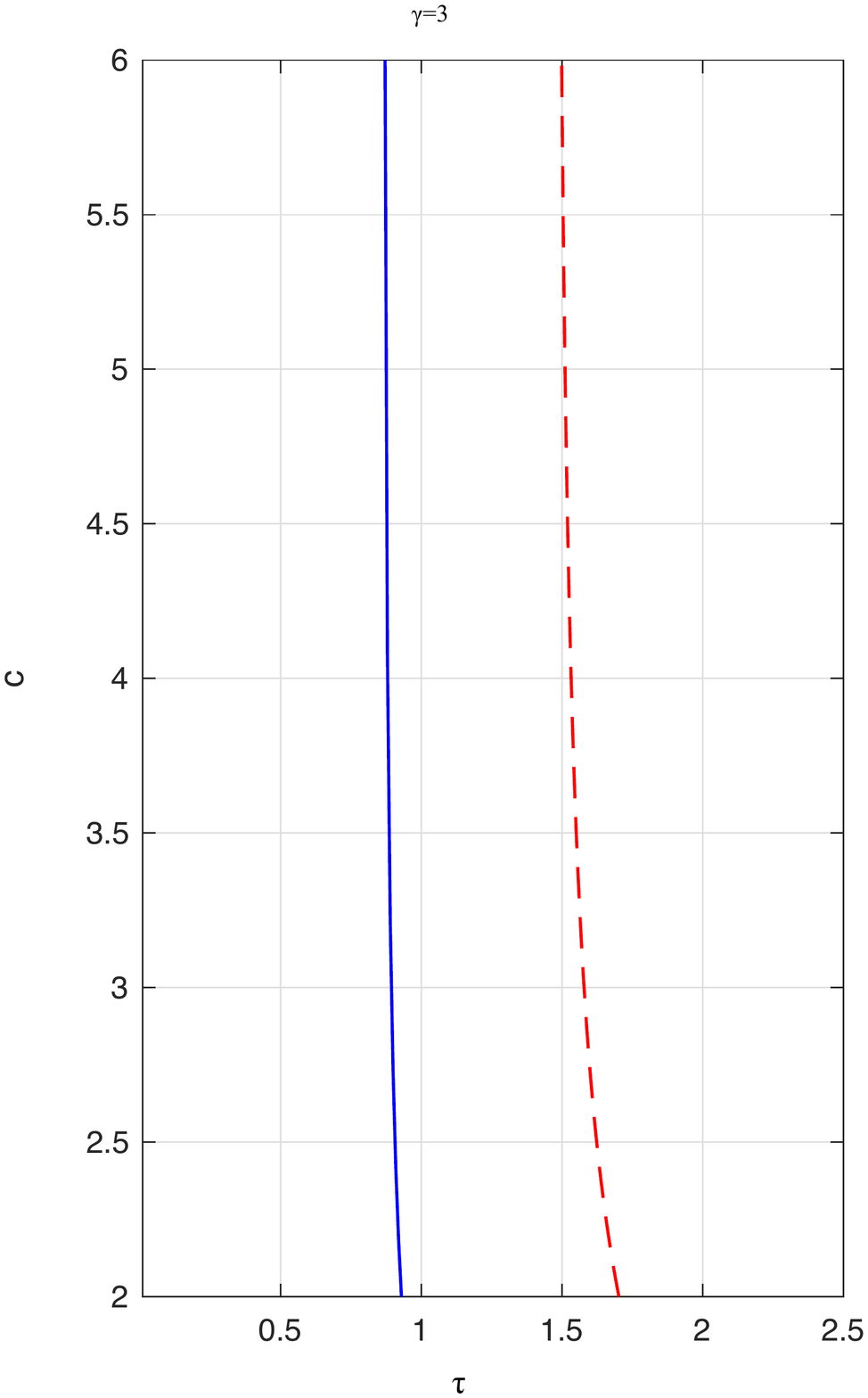}}}
\caption{The domains bounded by the blue (respectively, red) lines and the coordinate axes correspond to the sufficient (respectively, necessary) conditions for the existence of monotone waves.  Here we take $\gamma \in \{0.1, 1, 3\}$. }
\label{stab1}
\end{figure}
\end{example}
The next main result of this paper answers  the  question about the uniqueness of monotone waves.
\begin{thm}\label{T12}  Suppose that $\phi(t), \psi(t)$ are two monotone wavefronts to
 equation (\ref{i1}) propagating with  the same speed.  Then there exists a  real number $t'$ such that $\phi(t)\equiv \psi(t+t')$.  \end{thm}
 It is important to stress that the above uniqueness result is valid only within the sub-class of  monotone wavefronts. As it was shown in \cite{HK}, the uniqueness property does not hold within the larger class of all wavefronts even for the {\it nonlocal}\footnote{But it does hold for the delayed {\it local} KPP-Fisher equation, see \cite{ST}. } KPP-Fisher equation (i.e. when $\gamma =0$). 
 
Theorem \ref{T12} is a  non-trivial extension of the uniqueness result from \cite{FZ} established for the case $\gamma=0$. Indeed,  the proof in \cite{FZ} (see also \cite{GT})
uses  in essential way  the sub-tangency property of $F(u,v)$ (given by (\ref{F})) at the equilibrium $1$. In particular, this property (valid only for $\gamma =0$) is needed to identify the first term of   asymptotic representation of the wavefront profile at $+\infty$.  Hence, in order to prove Theorem \ref{T12} (where $\gamma >0$), it was necessary to find a completely different way to the uniqueness  problem.  In part, our approach was suggested by the recent study \cite{ST} proposing a new method for tackling the wave uniqueness problem for the delayed differential reaction-diffusion equation. However,  the proofs in \cite{ST} work only for the delayed diffusive equations with local interaction. To get rid of this restriction imposed in \cite{ST}, here  we are considering only monotone wavefronts  and  combine the main ideas from \cite{ST} with the sliding solution method of   Berestycki and  Nirenberg  \cite{cdm}.

Finally, a few words concerning the organization of the paper. Theorem \ref{T11} is proved in Section 2.   Section 3 contains the proofs of Theorem \ref{T12} as well as several other  auxiliary assertions.   

\section{Monotone wavefronts: the existence}
\subsection{Sufficiency conditions}
In this subsection, using the monotone iteration techniques developed in \cite{FZ,GT,wz}, we prove that equation  (\ref{i1}) has a monotone wavefront propagating with the  speed $c'\geq 2$ in the case when the characteristic equation $\chi_+(z,c',\tau)=0$ has a negative root $z'$  satisfying the inequality 
{$\gamma(z'^2 -cz') \leq 1$ }.  In what follows, we will omit the prime symbol in  $c'$ and write $\chi_+(z',c,\tau)=0$. 
 Next, $C_b(\R,\R)$ will denote the space of all bounded continuous functions (its topology will be chosen later). 

First, we observe that equation (\ref{f1}) can be rewritten in the form 
\begin{equation}\label{f1e}
 \phi''(t) - c\phi'(t) +\phi(t) =({\mathcal F}\phi)(t), 
\end{equation}
where operator ${\mathcal F}: C_b(\R,\R) \to C_b(\R,\R)$ defined by 
$$ ({\mathcal F}\phi)(t) = \phi(t)\frac{(1+\gamma)(N_c *\phi)(t)}{1+\gamma (N_c*\phi)(t)} $$
is clearly monotone non-decreasing in $\phi$.  Now, suppose that $c \geq 2$, then the equation $z^2-cz+1=0$ has two positive roots $0< z_1\leq z_2$ (counting multiplicity).  Furthermore,  it is easy to see \cite{FZ,GT} that 
every non-negative bounded solution of (\ref{f1e}) should satisfy 
the integral equation 
\begin{equation}\label{fie}
\phi(t) = \int_0^{+\infty}A(s) ({\mathcal F}\phi)(t+s)ds =: (\mathcal{A}\phi)(t), 
\end{equation}
where $$ A(s) = \left\{
\begin{array}{ccc}
a(e^{-z_1 s}- e^{-z_2 s})&,& s \geq 0\  \mbox{and}\ c >2,\\
se^{- s}&,& s\geq 0\  \mbox{and}\ c =2,\\
0 &,& s \leq 0, \ c \geq 2, 
\end{array}, \right. $$
with positive $a$  chosen to assure the normalization condition $\int_\R A(s)ds =1$.

Next, we  consider the following  modification of  equation  (\ref{f1}): 
\begin{equation}\label{f1m}
 \phi''(t) - c\phi'(t) +\phi(t)\left(\frac{1-(N_c *\phi)(t)}{1+\gamma}\right)=0.
\end{equation}
By our assumptions, for the fixed speed $c=c'$, the integral 
$$
\frak{I}_{c,\tau}(\lambda) = \int_\R N_c(s,\tau)e^{-\lambda s}ds
$$
is converging for $\lambda =z'<0$ and therefore $\Lambda= (-\lambda_0(c),0] \supset [z',0]$. Therefore \cite[Theorems 1.1, 1.2]{FZ} imply the existence of a unique (up to  translation) monotone wavefront solution $\phi_+(t)$ to equation (\ref{f1m}). Observe here that $c' \geq 2 > c_*:= 2/\sqrt{1+\gamma}$, with $c_*$ being the minimal speed of propagation in (\ref{f1m}). 
Clearly,   
\begin{equation}\label{f1u}
 \phi''_+(t) - c\phi'_+(t) +\phi_+(t)\left(\frac{1-(N_c *\phi_+)(t)}{1+\gamma (N_c *\phi_+)(t)}\right)\geq 0, 
\end{equation}
so that $(\mathcal{A}\phi_+)(t) \leq \phi_+(t), \ t \in \R,$ (i.e. $\phi_+(t)$ is an upper solution for equation  (\ref{f1})). 

Now, since $\chi_+(z',c,\tau)=0,$ $\chi_+(0,c,\tau) <0,$ $\chi_+(-\lambda_0(c)+,c,\tau) =-\infty,$ and  $\chi^{(4)}_+(z,c,\tau) <0, \ z \in \Lambda$, the  function  
$\chi_+(z,c,\tau)=0$ has at  most four negative zeros.   Without  loss of generality we can assume that $z'$ is the maximal negative zero so that $\chi_+(z,c,\tau)<0$
for $z \in (z',0]$. 
\subsubsection{Non-critical case: $c>2$ and $z'$ is a simple zero of $\chi_+(z,c,\tau)$ satisfying \\ $\gamma(z'^2 -cz') \leq 1$. }\label{S212}
Assuming the above restrictions,   we will consider the following function
$$ \phi_-(t) = \left\{
\begin{array}{ccc}
1- e^{z' t} &,& t \geq \zeta \\
a e^{z_1 t}&,& t \leq \zeta
\end{array}, \right. $$
where positive $a$ and $\zeta \in \R$ are chosen to assure the continuity of the derivative $\phi_-'(t)$ on $\R$. 
Taking into account the opposite convexities of these two pieces of $\phi_-$, it is easy to deduce 
the existence of such $a, \zeta$; in addition, due to  the same reason,
\begin{equation}\label{aux}
1- e^{z' t} < a e^{z_1 t}, \quad t < \zeta.
\end{equation}
Note that $\phi_+(t)$ and $\phi_-(t)$ have similar  asymptotic behavior at $+\infty$ and  $\phi_+(t)$ decreases more slowly than  $\phi_-(t)$  as $t \to -\infty$. 
 Indeed, since $z'$ is a simple zero of $\chi_+(z,c,\tau)$, invoking \cite[Lemma 3.2]{FZ}, we find that 
$$
\phi_+(t) = 1- e^{z't}(\theta_+ + o(1)), \ t \to +\infty, 
$$
for some $\theta_+ >0$. On the other hand, since $c>2> 2/\sqrt{1+\gamma}$ and reaction term of  equation (\ref{f1m}) satisfies the following sub-tangency   condition at $0$: 
$$\phi(t)\left(\frac{1-(N_c *\phi)(t)}{1+\gamma}\right)\leq \frac{\phi(t)}{1+\gamma},$$
we may conclude (cf. \cite[Lemma 28]{GT}) that, for some $\theta_- >0$, it holds 
$$
\phi_+(t) = e^{z_\gamma t}(\theta_- + o(1)), \ t \to -\infty, 
$$
where $z_\gamma < z_1$ is the minimal positive root of equation $z^2-cz+1/(1+\gamma)=0$.
Therefore there exists some non-negative $\sigma$ such that $\phi_+(t+\sigma) > \phi_-(t)$ for all $t \in \R$, cf. \cite[Lemma 18]{GT}.
Clearly, without loss of generality, we can suppose that $\sigma =0$ so that 
\begin{equation} \label{pp}
\phi_-(t) < \phi_+(t), \quad t \in \R. 
\end{equation}
The above form of function  $\phi_-$ was suggested in  \cite{FZ,GT}. It follows from the definition of $\zeta$ that 
$$
1- e^{z' \zeta} = a e^{z_1 \zeta},
$$
so that $e^{z' \zeta} \in (0,1)$.  Therefore,  for all $t \geq \zeta$, 
$$
1+\gamma (N_c *\phi_-)(t) \geq 1+\gamma (N_c *(1-e^{z' \cdot} )(t) = 1 +\gamma(1- (1+\gamma)(z'^2-cz')e^{z't}) \geq 
$$
\begin{equation}\label{zn}
(1 +\gamma)(1- \gamma (z'^2-cz')e^{z'\zeta}) > (1 +\gamma)(1- \gamma (z'^2-cz')) \geq 0.  
\end{equation}
\begin{lemma}\label{US}  Assume  that  $\gamma(z'^2 -cz') \leq 1$ then 
\begin{equation}\label{f1l}
I_-(t):= \phi''_-(t) - c\phi'_-(t) +\phi_-(t)\left(\frac{1-(N_c *\phi_-)(t)}{1+\gamma (N_c *\phi_-)(t)}\right)\leq 0, \ t \not= \zeta, 
\end{equation}
i.e. $\phi_-(t)$ is a  lower solution for  (\ref{f1}), $(\mathcal{A}\phi_-)(t) \geq \phi_-(t), \ t \in \R$.  
\end{lemma}
\begin{proof}  Since (\ref{f1l}) can be written as 
$$
 \phi''_-(t) - c\phi'_-(t) +\phi_-(t) \leq \phi_-(t)\frac{(1+\gamma)(N_c *\phi_-)(t)}{1+\gamma (N_c*\phi_-)(t)} $$
 and  $e^{z_1 t}$ is an eigenfunction for the linear equation $y''(t) -cy'(t) + y(t)=0$, relation (\ref{f1l}) clearly holds for all $t < \zeta$.

Now,  using (\ref{aux}) we find  that, for all $t \in \R$, 
$$
1-(N_c *\phi_-)(t) \leq 1-(N_c *(1-e^{z'\cdot })(t) = (N_c *e^{z'\cdot })(t) = (1+\gamma)(z'^2-cz')e^{z't}.
$$
Note that $z'^2-cz' >0$. Therefore,  in view of  (\ref{zn}), for $t >\zeta$, we obtain that
\begin{eqnarray}
\frac{e^{-z't}I_-(t)}{z'^2-cz'}
&\leq& -1 + \frac{[1- e^{z' t}](1+\gamma)}{1 +\gamma(1- (1+\gamma)(z'^2-cz')e^{z't})}\nonumber\\
&=& -1 + \frac{1- e^{z' t}}{1- \gamma(z'^2-cz')e^{z't}}=  \frac{e^{z' t}(\gamma(z'^2-cz')-1)}{1 - \gamma(z'^2-cz')e^{z't}}
\leq 0. \nonumber
\end{eqnarray}
This completes the proof of Lemma \ref{US}. \end{proof}
Since operator $\mathcal{A}$ is monotone, (\ref{pp}) implies that 
$$
\phi_- \leq \mathcal{A}\phi_- \leq \mathcal{A}\phi_+ \leq \phi_+. 
$$
Similarly, since $\phi_\pm(t) \leq \phi_\pm(t+s)$ for each $s>0$, we conclude that $\mathcal{A}\phi_\pm(t)$ are nondecreasing functions on $\R$. 
In this way, the sequence $\phi_n(t) = \mathcal{A}^n\phi_+, \ n =0,1,2\dots $ consists from non-decreasing continuous functions satisfying 
$$
\phi_-(t) \leq \phi_{n+1}(t)  \leq \phi_{n}(t)  \leq \phi_+(t), \ t \in \R. 
$$
Let $\phi_*(t) =\lim_{n\to +\infty} \phi_n(t)$. Then $\phi_*(t)$ is nondecreasing on $\R$ and $\phi_*(-\infty)=0$, $\phi_*(+\infty)=1$. 
Next, since $\phi_{n+1} = \mathcal{A}\phi_n, $ the obtained limit function satisfies 
$
\phi_*= \mathcal{A}\phi_* 
$  in virtue of Lebesgue's dominated convergence theorem. 
This means that $\phi_*(t)$ is a wavefront profile.  Since $\phi_*(t) \in (0,1), \ t \in \R$,    it is easy to see that  $\phi'_*(t) >0$   for all $t\in \R$, see also Lemmas \ref{mono} and \ref{USAI} below.
\subsubsection{Critical case: either $c=2$ or $z'$ is a multiple zero of $\chi_+(z,c,\tau)$ satisfying \\ $\gamma(z'^2 -cz') \leq 1$.}
First, we consider the case when $\gamma(z'^2 -cz') < 1$. Then the continuity of  $\frak{I}_{c,\tau}(\lambda)$  and a direct geometric analysis of the graph of $\chi_+$ show that there exist strictly decreasing sequences $\gamma_j \to \gamma$ and  $c_j \to c$ such that 
$\gamma_j(z_j'^2 -c_jz_j') < 1$ where $z_j$ is the (maximal) negative simple zero  of $\chi_+(z,c_j,\tau)$.  By the result of Subsection \ref{S212}, equation (\ref{f1})  $\gamma =\gamma_j$ and  $c= c_j$  has a monotone wavefront $\phi_j$ for each  $j$.  Without loss of generality, we can
assume  the condition $\phi_j(0)=0.5$. In addition, we conclude from  equation (\ref{f1})  
that the derivatives $\phi_j'(t)$ are uniformly bounded, 
$
0 < \phi'_j(t) \leq 2/c, 
$
$t\in \R$, $j\in \N$. 
Therefore the functional sequence $\{\phi_j(t)\}$ converges, uniformly on bounded sets, to some nondecreasing non-negative bounded  function $\phi_*(t),$ $\phi_*(0)=0.5$. 
Taking limit,  as $j \to +\infty$, in an appropriate integral form of the differential equation
$$
 \phi''_j(t) - c_j\phi'_j(t) -\phi_j(t) = - 2\phi_j(t) +({\mathcal F}\phi_j)(t), 
$$
we find that $\phi_*(t)$ also satisfies (\ref{f1}).  Since $\phi_*(0)= 0.5$ and $\phi_*(t)$ is a non-negative and monotone function,  we conclude that
$\phi_*(+\infty)= 1,$ $ \ \phi_*(-\infty)=0$. This completes the analysis of the critical case when $\gamma(z'^2 -cz') <1$.

Next, fix $c,\gamma, \tau$ and suppose that  $\gamma(z'^2 -cz') = 1$.  Also consider a sequence of shifted kernels $M_j(v,\tau):= N_c(v+1/j,\tau)$ together with the problem
\begin{equation}\label{f1M}
 \phi''(t) - c\phi'(t) +\phi(t)\left(\frac{1-(M_j *\phi)(t)}{1+\gamma (M_j*\phi)(t)}\right)=0, \quad  \phi(-\infty) = 0 , \;\; \phi( +\infty)=1. 
\end{equation}
Clearly, $M_j$ has the same general properties as $N_c$ while 
$$
(M_j*\phi)(t)= \int_\R N_c(v,\tau)\phi(t-v+1/j)dv
$$
and 
the characteristic function $\chi_+$ at the positive equilibrium takes the form
$$
\chi_{+,j}(z,c,\tau) = z^2-cz - \frac{e^{z/j}}{1+\gamma} \int_\R N_c(v,\tau)e^{-z v}dv. 
$$
Since $e^{z/j} <1$ for negative $z$, equation $\chi_{+,j}(z,c,\tau) =0$ has a negative root $z'_j$ satisfying the relation 
$\gamma(z'^2_j -cz')<1$  for all positive integer $j$. By the first part of this subsection, this implies the existence of strictly monotone solution $\phi_j(t)$ of (\ref{f1M})
for each $j$.  As before, we assume that $\phi_j(0)=0.5$ and that $\{\phi_j(t)\}$ converges, uniformly on compact subsets of $\R$, to a non-decreasing function $\phi_*(t)$.   
Thus we can use the Lebesgue dominated convergence theorem to establish the limit 
$$
(M_j*\phi_j)(t)= \int_\R N_c(v,\tau)\phi_j(t-v+1/j)dv \to
(M_c*\phi_*)(t)= \int_\R N_c(v,\tau)\phi_*(t-v)dv
$$
for each $t \in \R$. In order to complete the proof, we can now argue as in the first part of this subsection.
\subsection{Necessity of the condition imposed on $\chi_+(z,c,\tau)$}\label{S22}
Let $\phi$ be a positive monotone wavefront of equation (\ref{f1}). Then equation   (\ref{f1}) can be rewritten as 
\begin{equation}\label{R}
\phi''(t)-c\phi'(t) + (1-R(t))\phi(t)=0, 
\end{equation}
where  the function 
$$R(t)=\frac{(1+\gamma)(N_c *\phi)(t)}{1+\gamma (N_c*\phi)(t)}, \quad R(-\infty)= 0, $$
is also monotone on $\R$ and satisfies $\int_{-\infty}^0 |R'(s)|ds = R(0)<\infty$.  This allows the use of  the Levinson asymptotic integration theorem  showing  that the speed of propagation $c$ should satisfy the inequality  $c\geq 2$, cf. the proof of Lemma 18 in \cite{GT}. 

Next, our proofs in this subsection simplify  when the support supp\,$N_c$ of $N_c$ belongs to $(-\infty,0]$. Then the characteristic function $\chi_+$ takes the form  
$$
\chi_+(z,c,\tau) = z^2-cz - \frac{1}{1+\gamma} \int_{-\infty}^0 N_c(v,\tau)e^{-z v}dv,
$$
so that $
\chi_+(0,c,\tau)<0$ and $\chi_+(-\infty,c,\tau)=+\infty$. In consequence, $\chi_+(z,c,\tau)$ has at least one negative zero if  supp\,$N_c \subset (-\infty,0]$.  Thus we have to analyze 
only the situation when supp\,$N_c \cap (0, +\infty) \not=\emptyset$, i.e. when there exists  $r_0 >0$ that for each $s \in \R$ and 
positive $\nu = \int^{+\infty}_{r_0}N_c(r,\tau)dr \in (0,1)$, it holds
\begin{eqnarray}
\label{1N}
& & 1-(N_c*\phi)(s) = \int_\R(1-\phi(s-r))N_c(r,\tau)dr \geq  \nonumber \\
& &  \int^{+\infty}_{r_0}(1-\phi(s-r))N_c(r,\tau)dr\geq \nu (1-\phi(s-r_0)).
\end{eqnarray}
Our subsequent analysis is inspired by the arguments proposed in \cite{FZ} and \cite{HT},  we present them here for the sake of completeness. 

Set $y(t)=1-\phi(t)$, where $\phi(t)$ is a positive monotone wavefront. Then (\ref{fie}) takes the form
$$
y(t) = \int_0^{+\infty}A(s)\frac{(1+\gamma)(N_c*\phi)(t+s)y(t+s)+ (N_c*y)(t+s) }{1+\gamma (N_c*\phi)(t+s)}ds. 
$$
In view of (\ref{1N}), this implies that, for all $t \in \R$,
$$
y(t) \geq  \int_0^{+\infty}A(s)\nu \frac{y(t+s-r_0) }{1+\gamma}ds  \geq  \int_0^{r_0/2}A(s) \frac{\nu}{1+\gamma}y(t+s-r_0)ds\geq \hat \nu y(t-r_0/2),
$$
where $\hat \nu =  \nu\int_0^{r_0/2}A(s)ds/({1+\gamma})$. 
Therefore, for some $C>0$ and $\sigma= 2r_0^{-1}\ln \hat \nu <0$,  
\begin{equation}\label{rha}
y(t) \geq Ce^{\sigma t}, \quad t \geq 0. 
\end{equation}
Hence, 
\begin{equation}\label{si}
0 \geq \sigma_* = \liminf_{t\to +\infty} \frac 1 t \ln y(t) \geq \sigma.  
\end{equation}
Suppose, on the contrary, that $\chi_+(z,c,\tau)$ has not negative zeros. Then $\chi_+(\sigma_*,c,\tau)<0$ (we admit here the situation when $\chi_+(\sigma_*,c,\tau)=-\infty$), so that there exist a large $n_0>0$ and small $\delta, \rho >0$ such that 
$$
q:=\inf_{x \in (\sigma_*-\delta, \sigma_*+\delta)} \frac{(1+\gamma)(1-\rho) +\int_{-n_0}^{n_0} N_c(v,\tau)e^{-x v}dv}{(1+\gamma)(x^2-cx +1)}  >1.
$$
Next, let $\mu \in  (\sigma_*-\delta, \sigma_*)$ be such that 
$$\mu+ n_0^{-1}\ln q > \sigma_*. $$
Clearly, for some $C_1 >0$, it holds  that 
$$
y(t) \geq C_1e^{\mu  t}, \quad t \geq -n_0. 
$$
Since $y(+\infty) =0$, without loss of generality, we can assume that $(N_c*y)(s) < \rho$ for all $s \geq 0$, then $(N_c*y)(t+s) < \rho$ for all $t, s \geq 0$ and 
$$
y(t) \geq \int_0^{+\infty}A(s)\frac{(1+\gamma)(1- \rho)y(t+s)+ (N_c*y)(t+s) }{1+\gamma}ds \geq 
$$
$$
C_1 e^{\mu t} \int_0^{+\infty}A(s)e^{\mu s}\frac{(1+\gamma)(1- \rho)+ \int_{-n_0}^{n_0} N_c(v,\tau)e^{-\mu  v}dv }{1+\gamma}ds=
$$
$$
C_1 \frac{e^{\mu  t}}{\mu^2-c\mu +1} \frac{(1+\gamma)(1- \rho)+ \int_{-n_0}^{n_0} N_c(v,\tau)e^{-\mu v}dv }{1+\gamma} \geq C_1q e^{\mu  t},  \quad t \geq 0. 
$$
Repeating the same argument for $y(t)$ on the interval $[n_0,+\infty)$, we find similarly that $y(t) \geq C_1q^2 e^{\mu  t},  \ t \geq n_0.$  Reasoning in this way, 
we obtain the estimates
$$y(t) \geq C_1q^{j+1} e^{\mu  t}\geq C_1e^{(\mu+n_0^{-1}\ln q)t},  \ t \in [n_0j,n_0(j+1)], \quad j=0,1,2\dots 
$$
This yields the following contradiction:  
$$
\sigma_* = \liminf_{t\to +\infty} \frac 1 t \ln y(t) \geq \mu+n_0^{-1}\ln q > \sigma_* . 
$$
As a by product of the above reasoning, we also get the following statement: 
\begin{lemma}\label{bp} Suppose that supp\,$N_c \cap (0, +\infty) \not=\emptyset$ and let $\phi(t)$ be a positive monotone wavefront.  Set $y(t)=1-\phi(t)$ and let $\sigma_*$ be defined as in (\ref{si}). Then $\chi_+(\sigma_*,c,\tau)$ is finite and  non-negative. In particular,  $-\lambda_0(c) < \sigma_* <0$ and  $\chi_+(x,c,\tau)$ has at least one zero on the interval $[\sigma_*,0]$. 
\end{lemma}
\begin{remark} \label{Re1} Lemma \ref{USA} below improves further  the result of Lemma \ref{bp}.  Next, let $\{z: \Re z > \alpha(\phi)\}\subset \C$ be the maximal open  strip where the Laplace transform $\tilde y(\lambda)$ of $y(t)$ is defined. Since $y(t)$ is bounded on $\R$, we have that $ \alpha(\phi) \leq 0$. On the other hand, by the definition of $\sigma_*$,  it is easy to see 
that $\lim_{t\to+\infty}y(t)e^{-\lambda t} =+\infty$ for every $\lambda < \sigma_*$. Thus $ \alpha(\phi) \geq  \sigma_* > -\lambda_0(c)$. Note also that $\alpha(\phi)$ is a singular point 
of $\tilde y(\lambda)$, cf. \cite{LT}. 
\end{remark}
\section{Monotone wavefronts: the uniqueness}
\subsection{Three auxiliary results}
\begin{lemma}\label{mono} Suppose that $\phi(t) \in [0,1], \ t \in \R$,  is a non-constant solution of equation (\ref{f1}). Then  $\phi(t) >0$ for all $t$ and there exists $s' \in \R \cup\{+\infty\}$ such that  $\phi'(t) >0$ for all $t < s'$ and $\phi(t)=1$ for all $t \geq s'$.  Moreover, if  $s'=+\infty$ if supp $N_c \cap (0,+\infty)\not=\emptyset$. 
 \end{lemma}
\begin{proof} Equation   (\ref{f1}) can be rewritten as $\phi''(t)-c\phi'(t) + \omega(t)\phi(t)=0$, where $$\omega(t):= \left(\frac{1-(N_c *\phi)(t)}{1+\gamma (N_c*\phi)(t)}\right)\geq 0, \quad t \in \R,$$
is a continuous bounded function. Suppose that  $\phi(s)=0$ for some $s$. Then non-negativity of $\phi$ implies that 
$\phi'(s)=0$. 
Therefore, in view of the existence and uniqueness theorem for linear ordinary  differential equations, we have that $\phi \equiv 0$ which contradicts our assumption.  Thus 
$\phi(t) >0$ for all $t \in \R$. 

Next, integrating the above equation with respect to $\phi'(t)$, we find easily that 
$$
\phi'(t)=\int_t^{+\infty}e^{c(t-s)}\omega(s)\phi(s)ds \geq 0.$$
This immediately  implies  the existence of $s'$ with the above indicated properties. Now, if $s'$ is finite then $w(t)=0$ for all $t \geq s'$ so that $(N_c *\phi)(t)=1$ for all 
$t\geq s'$.  Evidently, this can happen if and only if $\phi(t) =1$ for all $t \geq s'$ and  supp $N_c \cap (0,+\infty)=\emptyset$. 
\end{proof}
In fact, as we will see from the next lemma, $s' =+\infty$ for every admissible $N_c$. 
\begin{lemma}\label{USAI} Suppose that supp $N_c \cap (0,+\infty)=\emptyset$ and  let $\phi(t)$ be a monotone wavefront to  equation (\ref{f1}). Then $y(t)=1-\phi(t)$
satisfies 
\begin{equation}\label{areI}
y(t) \geq y(s)e^{(s-t)/(c(1+\gamma))} \quad \mbox{for all} \ t \geq s. 
\end{equation}
\end{lemma}
\begin{proof} We have 
\begin{equation}\label{efy}
y''(t)-cy'(t) - \phi(t)\frac{y*N_c(t)}{1+\gamma \phi*N_c(t)}=0, \quad t \in \R. 
\end{equation}
Clearly, $\phi*N_c(t) = \int_{-\infty}^0\phi(t-s)N_c(s)ds \geq \phi(t)\int_{-\infty}^0N_c(s)ds=\phi(t), \ t \in \R$, so that 
$$
y(t) = 1-\phi(t) \geq 1-\phi*N_c(t)= y*N_c(t), \quad   \frac{\phi(t)}{1+\gamma \phi*N_c(t)}\leq \frac{1}{1+\gamma},\quad t \in \R. 
$$
Using the notation
\begin{equation}\label{rr}
y'(t) =z(t), \quad r(t) =  \phi(t)\frac{y*N_c(t)}{1+\gamma \phi*N_c(t)} -  \frac{y*N_c(t)}{1+\gamma}\leq 0, \quad t \in \R,
\end{equation}
we find  that 
$$
z'(t)=cz(t) + \frac{y*N_c(t)}{1+\gamma} + r(t), \ t \in \R. 
$$
Since $z(\pm\infty)=0$, we also have 
$$
y'(t)=z(t) = -\int_t^{+\infty}e^{c(t-s)}\left(\frac{y*N_c(s)}{1+\gamma} + r(s)\right)ds\geq 
$$
$$
 -\int_t^{+\infty}e^{c(t-s)}\frac{y*N_c(s)}{1+\gamma}ds\geq  -\int_t^{+\infty}e^{c(t-s)}\frac{y(s)}{1+\gamma}ds=-y(t)\frac{1}{c(1+\gamma)}, \quad t \in \R. 
$$
Thus 
$$
(y(t)e^{t/(c(1+\gamma))})'\geq 0, \quad t \in \R, 
$$
which implies (\ref{areI}). 
\end{proof} 
\begin{lemma}\label{USA} Let $\phi(t)$ be a monotone wavefront to  equation (\ref{f1}). Then there exist $t_1, t_2 \in \R$  such that  
\begin{equation}\label{are}
\phi(t+t_1) = (-t)^je^{z_1t}(1+o(1)), \ t \to -\infty, \quad \phi(t+t_2) = 1-t^ke^{\hat zt}(1+o(1)), \ t \to +\infty. 
\end{equation}
where $j=0$ if $c >2$ and $j=1$ when $c=2$; $k \in \{0,1,2,3\}$ and  $\hat z=\hat z(\phi)$ is some negative root of the characteristic equation $\chi_+(z,c,\tau)=0$.\end{lemma}
\begin{proof} \underline{Asymptotic representation of $\phi$ at $+\infty$}. 
Our first step is to establish that $y(t)=1-\phi(t)$ has an exponential rate of convergence to $0$ at $+\infty$. 
To prove this, we will need the next property:  
\begin{equation} \label{rh}
\mbox{For some}\ \rho >0 \ \mbox{it holds that} \quad (y*N_c)(t) \geq \rho y(t), \quad  t \in \R.
\end{equation}
Again, we will distinguish between two following situations.

\vspace{2mm}

{\sf Case 1:} supp $N_c \cap (0,+\infty)\not=\emptyset$.  Then there exists $m >0$ such that $\rho_1:= \int_0^mN_c(s,\tau)ds >0$ and 
$$
\int_\R y(t-s)N_c(s,\tau)ds \geq \int_0^my(t-s)N_c(s,\tau)ds \geq  \rho_1 y(t), \quad t \in \R.
$$

{\sf Case 2:} supp $N_c \cap (0,+\infty)=\emptyset$.  Then, by Lemma \ref{USAI},  we have,  for $t \in \R$,  
$$
(y*N_c)(t) = \int_{-\infty}^0y(t-s)N_c(s,\tau)ds \geq  y(t) \int_{-\infty}^0e^{s/(c(1+\gamma))}N_c(s,\tau)ds=:\rho_2 y(t).
$$
In every case, (\ref{rh}) holds with $\rho = \min\{\rho_1,\rho_2\}$. 

\vspace{2mm}

Next, since $\phi(+\infty)=1$, we can indicate $T_0$ sufficiently large to satisfy
$$
\frac{\phi(t)}{1+\gamma \phi*N_c(t)} > 0.5/(1+\gamma), \quad t \geq T_0. 
$$ 
With the positive number  
$$
\kappa = \frac{0.5\rho}{1+\gamma},
$$
we can rewrite equation (\ref{efy}) as 
$$
y''(t)-cy'(t) - \kappa y(t) = h(t), \quad \mbox{where} \ h(t):= \phi(t)\frac{y*N_c(t)}{1+\gamma \phi*N_c(t)} - \kappa y(t), \ t \in \R. 
$$
Importantly, for $t \geq T_0$, 
$$
h(t) > y*N_c(t)\left(\frac{0.5}{1+\gamma} - \frac{\kappa}{\rho}\right)=0, 
$$
so that, arguing as in \cite[Lemma 20, Claim I]{GT}, we conclude that 
$$
y(t) \leq y(s)e^{l(t-s)}, \quad t \geq s \geq T_0, \quad \mbox{where} \ \alpha(\phi) \leq l:= 0.5\left(c-\sqrt{c^2+4\kappa}\right)<0.  
$$
Hence, by Remark \ref{Re1}, $\frak{I}_{c,\tau}(l)$ is a finite number.  
Combining the latter exponential  estimate with results of  Lemma \ref{USAI} (if supp $N_c \cap (0,+\infty)=\emptyset$) or inequality (\ref{rha}) (if supp $N_c \cap (0,+\infty)\not=\emptyset$), we conclude that $y(t)$ has an exponential rate of convergence at $+\infty$.  Moreover, $y'(t)$ has the same property because of the estimates
$$R(t):= \phi(t)\frac{y*N_c(t)}{1+\gamma \phi*N_c(t)}\leq  y*N_c(t) =  \int_{-\infty}^{t-T_0} N_c(s,\tau)y(t-s)ds+
$$
$$ \int^{+\infty}_{t-T_0} N_c(s,\tau)y(t-s)ds \leq  \int_{-\infty}^{t-T_0} N_c(s,\tau)y(T_0)e^{l(t-s-T_0)}ds+ \int^{+\infty}_{t-T_0} N_c(s,\tau)e^{-ls}e^{ls}ds \leq $$
$$ \leq  e^{l(t-T_0)} \int_{-\infty}^{t-T_0} N_c(s,\tau)e^{-ls}ds+  e^{l(t-T_0)} \int^{+\infty}_{t-T_0} N_c(s,\tau)e^{-ls}ds =  e^{l(t-T_0)}\frak{I}_{c,\tau}(l)
$$
and 
$$
y'(t)= -\int_t^{+\infty}e^{c(t-s)}R(s)ds\geq -\frak{I}_{c,\tau}(l)e^{-lT_0}\int_t^{+\infty}e^{c(t-s)}e^{ls}ds= \frac{\frak{I}_{c,\tau}(l)e^{-lT_0}}{l-c}e^{lt}.
$$
The latter representation of $y'(t)$ is deduced from (\ref{efy}) which also implies that 
\begin{equation}\label{efyk}
y''(t)-cy'(t) - \left(\frac{1}{1+\gamma}+\epsilon(t)\right){(y*N_c)(t)}=0, \quad t \in \R,
\end{equation}
where 
$$
\epsilon(t):= \frac{\phi(t)}{1+\gamma \phi*N_c(t)}- \frac{1}{1+\gamma} = \frac{\gamma y*N_c(t)-(1+\gamma)y(t)}{(1+\gamma)(1+\gamma \phi*N_c(t))} = O(e^{lt}), \ t \to +\infty.  
$$
Then, in view of Remark \ref{Re1}, an application of \cite[Lemma 22]{TAT} shows that $y(t) = w_0(t)(1+o(1)), \ t \to +\infty$,  where $w_0(t)$ is a non-zero eigensolution of the equation 
$w''(t)-cw'(t) - {(w*N_c)(t)}/(1+\gamma)=0$ corresponding to some its negative eigenvalue $\hat z$. As we have already mentioned, the multiplicity of $\hat z$ is less or equal to 4. This proves the second representation in (\ref{are}). 

\vspace{2mm}

 \noindent \underline{Asymptotic representation of $\phi$ at $-\infty$}.  Since the linear equation $y''-cy'+y=0$ with $c \geq 2$ is exponentially unstable, then so is equation (\ref{R}) with $R(-\infty)=0$ at 
 $-\infty$. This assures at least the exponential rate of convergence of $\phi(t)$ to $0$ at $-\infty$.  On the other hand, $\phi(t)$ has no more than exponential rate of decay 
 at $-\infty$, cf.  \cite[Lemma 6]{ST}.  Again, an application of \cite[Lemma 22]{TAT} shows that $y(t) = v_0(t)(1+o(1)), \ t \to +\infty$,  where $v_0(t)$ is a non-zero eigensolution of the equation 
$v''(t)-cv'(t) +v(t)=0$ corresponding to one of the positive eigenvalues  $z_1, z_2$. Finally, since the function $F(u,v)$ given in (\ref{F}) satisfies the sub-tangency condition at the zero equilibrium, we conclude that  the correct eigenvalue in our case is precisely $z_1$, see \cite[Section 7]{GT}  for the related computations and further details. 
 \end{proof} 
\begin{corollary}  \label{Com} Let $\psi(t), \phi(t)$ be {different} monotone wavefronts to  equation (\ref{f1}). 
Then there exist $t_3, t_4$ such that $\psi(t+t_3) \not= \phi(t+t_4)$ for all $t \in \R$ meanwhile $\phi(t+t_3),$ $\psi(t+t_4)$ have the same main asymptotic terms at $-\infty$:
$$
\phi(t+t_3) = (-t)^je^{z_1t}(1+o(1)), \ \psi(t+t_4) = (-t)^je^{z_1t}(1+o(1)), \ t \to -\infty. 
$$
\end{corollary}
\begin{proof} By Lemma \ref{USA}, there are $t_1, t_2, t_1', t_2', q \in \{0,1,2,3\},$ such that (\ref{are}) holds together with 
\begin{equation}\label{arek}
\psi(t+t_1') = (-t)^je^{z_1t}(1+o(1)), \ t \to -\infty, \quad \psi(t+t_2') = 1-t^qe^{z''t}(1+o(1)), \ t \to +\infty,
\end{equation}
where $z''=z''(\psi)$ is a negative root of the characteristic equation $\chi_+(z,c,\tau)=0$. After realizing appropriate translations of profiles, 
without loss of generality, we can assume that $t_1=t_1'=0$. 

Suppose that $\hat z\geq z''$, then there is sufficiently large  $A>0$ such that 
$$
\psi(t+A) > \phi(t), \quad t \in \R. 
$$
Since $\psi(t)$ is an increasing function,  for all $a \geq A$, 
$$
\psi(t+a) > \phi(t), \quad t \in \R. 
$$
Let $\frak A$ denote the set of all $a$ such that the latter inequality holds. Clearly, $\frak A$ is a below bounded set and therefore the number  $a_*= \inf \frak A$
is finite and 
$$
\psi(t+a_*) \geq \phi(t), \quad t \in \R. 
$$
Observe that,  since $\psi, \phi$ are different wavefronts, the difference $\delta(t)= \psi(t+a_*)-\phi(t)$ is a non-zero non-negative function satisfying 
$\delta(-\infty)=\delta(+\infty) =0$. 
We claim that actually $\delta(t) >0, \ t \in \R$, i.e. 
\begin{equation}\label{IN}
\psi(t+a_*) > \phi(t), \quad t \in \R. 
\end{equation}
Indeed, otherwise there exists some $s'$ such that $\delta(s') =0$. From (\ref{f1e}), we have that 
$$
 \delta''(t) - c\delta'(t) +\delta(t) = {\mathcal H}(t),  
$$
where
$$ {\mathcal H}(t) = ({\mathcal F}\psi)(t+a_*) - ({\mathcal F}\phi)(t) \geq 0, \quad t \in \R.$$
Now, similarly to (\ref{fie}), 
$$
\delta(t) = \int_0^{+\infty}A(s) {\mathcal H}(t+s)ds.  
$$
Since $A(s) >0$ for $s >0$ and $\delta(s')=0$, we get immediately that ${\mathcal H}(s) =0$ for all $s\geq s'$. Clearly, this means that 
$$
 \psi(t+a_*) = \phi(t) \ (\mbox{i.e.}\ \delta(t)=0), \quad \psi*N_c(t+a_*) = \phi*N_c(t), \quad t \geq s'.
$$
Next, suppose that supp $N_c \cap (0,+\infty)\not=\emptyset$ and that 
$[s',+\infty)$ with $s'\in \R$ is the maximal interval where $\delta(t)=0$. Then $\int_{s'-s}^{+\infty}\delta(u)du >0$ for every $s>0$.  Furthermore, since 
$$\int_\R\delta(t-s)N_c(s,\tau)ds = \delta*N_c(t)=0, \quad t\geq s',$$
we obtain the following contradiction:  
$$
0= \int_{s'}^{+\infty}dt \int_0^{+\infty}\delta(t-s)N_c(s,\tau)ds=  \int_0^{+\infty}N_c(s,\tau)ds\int_{s'-s}^{+\infty}\delta(u)du >0,
$$
Thus $s'=-\infty$ and $\psi(t+a_*) = \phi(t)$ for all $t\in \R$ contradicting to our initial assumption that $\phi$ and $\psi$ are different wavefronts. This proves (\ref{IN}) when 
supp $N_c \cap (0,+\infty)\not=\emptyset$. 

In what follows, to simplify the notation, we suppose that $a_*=0$.  We will use another method when supp $N_c \subseteq  (-\infty,0]$.  In such a case, both functions $v_1(t):= 1-\phi(-t)$ and $v_2(t):=1- \psi(-t)$ satisfy 
the  initial value problem $$
 v_{-s'}(\sigma) = 1-\phi(s'-\sigma)= 1- \psi(s'-\sigma), \ \sigma \leq 0,\quad v'(-s') = \phi'(s') = \psi'(s')$$
for the following functional differential equation with unbounded delay: 
\begin{equation}\label{ud}
 v''(t) + cv'(t) -(1-v(t))\left(\frac{\int_{-\infty}^0v(t+s)N_c(s,\tau)ds}{1+\gamma- \gamma \int_{-\infty}^0v(t+s)N_c(s,\tau)ds}\right)=0. 
\end{equation}
Due to the optimal nature of $s'$, the solutions  $v_1(t)$ and $v_2(t)$ do not coincide on intervals $(-s', -s'+\epsilon)$ for $\epsilon >0$. 
On the other hand, since the function $g(x,y)= (1-x)y/(1+\gamma(1-y)) $ is globally Lipschitzian in the square $[0,1]^2\subseteq \R^2$, we can use the standard 
argumentation\footnote{It suffices to rewrite (\ref{ud}) in an equivalent form of a system of integral equations and then, after some elementary transformations,  to apply the Gronwall-Bellman inequality.} to prove that, for all sufficiently small $\epsilon>0$,  $v_1(t)=v_2(t)$ for  $t\in (-s', -s'+\epsilon)$. Thus again we get a contradiction proving (\ref{IN}) when 
supp $N_c \cap (0,+\infty)=\emptyset$. 

\vspace{2mm}

Next, clearly, 
$$
\kappa:=\lim_{t\to -\infty}\psi(t)/\phi(t) \geq 1, \quad \lim_{t\to +\infty}(1-\psi(t))/(1-\phi(t)) \leq 1.
$$
If $\kappa=1$, then $\psi(t)$ and $\phi(t)$ have the same asymptotic behavior at $-\infty$ and the corollary is proved. So, let suppose  that  $\kappa >1$. 
Then the optimal nature of $a_*=0$ implies that $\psi(t)$ and $\phi(t)$ have the same asymptotic behavior at $+\infty$. Thus  $\hat z = z'',$ $ k=q$ and 
$$
\lim_{t\to +\infty}{(1-\psi(t))}/{(1-\phi(t))} =  1.  
$$
But then, for all sufficiently large positive $b$,  
$$
\phi(t+b) > \psi(t), \ t \in \R. 
$$
We now can argue as before to establish the existence of the minimal positive $b_*$ such that 
$$
\phi(t+b_*) > \psi(t), \ t \in \R. 
$$
Since $b_*>0$ we have that
$$
 \lim_{t\to +\infty}(1-\phi(t+b_*))/(1-\psi(t)))=  e^{\hat zb_*} <1.
$$
 Then the  optimal character of $b_*$  implies  that 
  $$
\lim_{t\to -\infty}\psi(t)/\phi(t+b_*) =  1. 
$$
This completes the proof of Corollary   \ref{Com} (where $\phi(t+t_3) := \phi(t+b_*)$ and $\psi(t+t_4): = \psi(t)$ should be taken)  in the case $\kappa >1$.
 \end{proof}
\subsubsection{Proof of Theorem \ref{T12}}
Suppose that there are two different wavefronts, $\phi(t)$ and $\psi(t)$ to equation (\ref{f1}).  By Corollary  \ref{Com}, without restricting the generality, we can 
assume that 
$$
w(t):=(\psi(t)-\phi(t))e^{-\lambda t} >0, \quad t \in \R, \quad  w(t) = e^{(z_2-\lambda)t}(A+o(1)),\  t \to -\infty,  \ \ c \geq 2.
$$
Take now some $\lambda \in (z_1,z_2)$ if $c >2$ and $\lambda=z_1=z_2=1$ if $c=2$. Then $w(t)$ is bounded on $\R$ and satisfies
the following equation for all $t \in \R$:
\begin{equation}\label{38}
w''(t)-(c-2\lambda)w'(t)+(\lambda^2-c\lambda+1)w(t) = e^{-\lambda t}\left( ({\mathcal F}\psi)(t) -  ({\mathcal F}\phi)(t) \right).
\end{equation}
Next, if $c >2$, then $z_2-\lambda >0$ and therefore $w(-\infty)=w(+\infty)=0$. This means that, for some $t^*$, 
$$
w(t^*)= \max_{s \in \R}w(s) >0, \ w''(t^*) \leq 0, \ w'(t^*) =0. 
$$
Then, evaluating (\ref{38}) at $t^*$ and noting that $\lambda^2-c\lambda+1 <0$,  $({\mathcal F}\psi)(t) >  ({\mathcal F}\phi)(t),$ $t \in \R$,  we get a contradiction in signs. 
This proves the uniqueness of all non-critical wavefronts. 

Suppose now that $c=2$, then equation (\ref{38})  takes the form 
$$
w''(t) =e^{-\lambda t}\left( ({\mathcal F}\psi)(t) -  ({\mathcal F}\phi)(t) \right) >0,\quad  t \in \R.
$$
Since $w(+\infty)=0$, this implies that $w'(t)<0$ for all $t \in \R$. Clearly, the inequalities $w'(t)<0, \ w''(t) >0$, $t \in \R$, are not compatible with the boundedness of  $w(t)$ at $-\infty$. This proves the uniqueness of the minimal wavefront. 

\section*{Acknowledgements}
This work was supported by FONDECYT (Chile), project 1170466 (E. Trofimchuk and M. Pinto). S. Trofimchuk was supported by  FONDECYT (Chile) project  1190712.


\end{document}